\documentclass[12pt,reqno]{amsart}
\usepackage{amssymb}

\newtheorem{theorem}{Theorem}[section]
\newtheorem{corollary}[theorem]{Corollary}

\theoremstyle{remark}
\newtheorem{remark}[theorem]{Remark}

\begin{document}

\baselineskip=15.5pt

\title[The full automorphism group of $\overline{T}$]{The full automorphism group 
of $\overline{T}$}

\author[I. Biswas]{Indranil Biswas}

\address{School of Mathematics, Tata Institute of Fundamental
Research, Homi Bhabha Road, Mumbai 400005, India}

\email{indranil@math.tifr.res.in}

\author[S. S. Kannan]{Subramaniam Senthamarai Kannan}

\address{Chennai Mathematical Institute, H1, SIPCOT IT Park, Siruseri,
Kelambakkam 603103, India}

\email{kannan@cmi.ac.in}

\author[D. S. Nagaraj]{Donihakalu Shankar Nagaraj}

\address{The Institute of Mathematical Sciences, CIT
Campus, Taramani, Chennai 600113, India}

\email{dsn@imsc.res.in}



\begin{abstract}
Let $\overline G$ be the wonderful compactification of 
a simple affine algebraic group $G$ of adjoint type defined over $\mathbb C.$
Let ${\overline T}\,\subset\,
\overline G$ be the closure of a maximal torus $T\, \subset\, G.$
We prove that the group of all automorphisms of the variety $\overline T$ is
the semi-direct product $N_G(T)\rtimes D,$ where $N_G(T)$ is the normalizer
of $T$ in $G$ and $D$ is the group of all automorphisms of the Dynkin diagram,
if $G \,\not=\, {\rm PSL}(2,\mathbb{C})$. Note that if 
$G \,=\, {\rm PSL}(2,\mathbb{C})$, then 
$\overline{T}\, =\, {\mathbb C}{\mathbb P}^1$ and so in this case
$\text{Aut}(\overline T)\,=\, {\rm PSL}(2,\mathbb{C})$. \\

\textsc{R\'esum\'e.}
\textbf{Le groupe complet des automorphismes de $\overline T$.}\,
Soit $\overline G$ la compactification magnifique d'un groupe alg\'ebrique affine simple 
$G$ de type adjoint d\'efini sur $\mathbb C$. Soit ${\overline T}\,\subset\,
{\overline G}$ la cl\^oture d'un tore maximal $T\,\subset\, G$. Si $G\not\,=\,
{\rm PSL}(2,{\mathbb C})$, nous montrons que le groupe de tous les automorphismes de la 
vari\'et\'e $\overline T$ est le produit semi-direct $N_G(T)\rtimes D$, o\`u $N_G(T)$ est
le normalisateur de $T$ dans $G$ et $D$ est le groupe de tous les automorphismes du
diagramme de Dynkin. Remarquez que si $G \,=\, {\rm PSL}(2,\mathbb{C})$, alors $\overline{T}
\,= \,{\mathbb C}{\mathbb P}^1$ et donc dans ce cas $\text{Aut}(\overline T) \,= \,
{\rm PSL}(2,\mathbb{C})$.
\end{abstract}

\maketitle

\section{Introduction}

Let $G$ be a simple affine algebraic group of adjoint type defined over the field of 
complex numbers. De Concini and Procesi constructed a very important compactification of 
$G$ \cite[p.~14, 3.1, THEOREM]{DP}; it is known as the wonderful compactification. The 
wonderful compactification of $G$ will be denoted by $\overline G$. Fix a maximal torus 
$T$ of $G$, and denote by $\overline T$ the closure of the variety $T$ in the wonderful 
compactification $\overline G$ \cite[\S~1]{BJ}. Let ${\rm Aut}(\overline{T})$ denote the 
group of all holomorphic automorphisms of $\overline T$. For $G \,\not=\, {\rm 
PSL}(2,\mathbb{C}),$ the connected component of ${\rm Aut}(\overline{T})$ containing the 
identity element coincides with $T$ acting on $\overline T$ by translations \cite[Theorem 
3.1]{BKN}. Our aim here is to compute the full automorphism group ${\rm 
Aut}(\overline{T})$.

It may be noted that $\overline{T}$ is stable under the conjugation of the
normalizer $N_{G}(T)$ of $T$ in $G$. This indicates that
${\rm Aut}(\overline{T})$ need not be connected.

For $G$ different from $ {\rm PSL}(2,\mathbb{C})$, we prove that ${\rm 
Aut}(\overline{T})$ is the semi-direct product $N_G(T)\rtimes D,$ where $N_G(T)$ is the 
normalizer of $T$ in $G,$ and $D$ is the group of all automorphisms of the Dynkin diagram 
(see Theorem \ref{thm2}).

\section{Lie algebra and algebraic groups}

We recall the set-up of \cite{BKN}. Throughout $G$ will denote an affine
algebraic group over $\mathbb{C}$ such that $G$ is simple and of adjoint type (equivalently,
the center of the simple group is trivial). We will always assume
that $G \,\not=\, {\rm PSL}(2,\mathbb{C})$.

Fix a maximal torus $T$ of $G$. The group of all characters of $T$ will be denoted by 
$X(T)$. The Weyl group of $G$ with respect to $T$ is defined to be $W\ :=\, N_{G}(T)/T$, 
where $N_{G}(T)$ is the normalizer of $T$ in $G$. Let
\begin{equation}\label{root}
R \,\subset \,X(T)
\end{equation} 
be the root system of $G$ with respect to $T$.
For a Borel subgroup $B$ of $G$ containing the maximal torus $T$, let 
$R^{+}(B)$ denote the set of positive roots determined by $T$ and $B$. Let
$$
S \,=\, \{\alpha_1\, ,\cdots\, ,\alpha_n\}
$$ 
be the set of simple roots in $R^{+}(B)$, where $n$ is the
rank of $G$. Let $B^{-}$ denote the opposite
Borel subgroup of $G$ determined by $B$ and $T.$ So in particular
$B\bigcap B^{-}\,=\, T$. For any $\alpha \,\in\, R^{+}(B)$, let 
$s_{\alpha} \,\in \, W$ be the 
reflection corresponding to $\alpha$.

The Lie algebras of
$G$, $T$ and $B$ will be denoted by $\mathfrak{g}$, $\mathfrak{t}$ and $\mathfrak{b}$
respectively. The dual of the real form $\mathfrak{t}_{\mathbb R}$ of $\mathfrak{t}$ is
$X(T)\otimes \mathbb{R}\,=\, {\rm Hom}_{\mathbb{R}}(\mathfrak{t}_{\mathbb{R}},\, \mathbb{R})$.

Now, let $\sigma$ be the involution of $G\times G$ defined by $\sigma(x\, ,y)\,=\,(y\, 
,x)$. We note that the diagonal subgroup $\Delta(G)$ of $G\times G$ is the subgroup of 
fixed points of $\sigma$. The subgroup $T\times T\, \subset\, G\times G$ is a 
$\sigma$--stable maximal torus of $G\times G$, while $B\times B^{-}$ is a Borel subgroup 
of $G\times G$; this Borel subgroup $B\times B^{-}$ has the property that 
$\sigma(\alpha)\,\in\, -R^{+}(B\times B^{-})$ for every $\alpha \,\in\, R^{+}(B\times 
B^{-}).$

The group $G$ is identified with the symmetric space $(G\times G)/\Delta(G)$.
Let $\overline{G}$ denote the corresponding wonderful compactification of $G$
(see \cite[p.~14, 3.1, THEOREM]{DP}).
In particular $G\times G$ acts on $\overline{G}$.
Let $\overline{T}$ be the closure of $T$ in $\overline{G}.$
The action of the subgroup $N_{G}(T)\, \subset\, G\,=\, \Delta(G)$ on
$\overline{G}$ preserves $\overline{T}$.

\section{The automorphism group of $\overline{T}$}

Let ${\rm Aut}(\overline{T})$ denote the group of all holomorphic automorphisms of
$\overline{T}$; any holomorphic automorphism is algebraic.
Let ${\rm Aut}^0(\overline{T})\, \subset\, {\rm
Aut}(\overline{T})$ be the connected component containing the identity element. 
The translation action of $T$ on itself produces an isomorphism 
\begin{equation}\label{rh}
\rho\, :\, T \,\longrightarrow\, {\rm Aut}^0(\overline{T})
\end{equation}
\cite[p.~786, Theorem 3.1]{BKN}.

\begin{theorem}\label{thm2}
The automorphism group ${\rm Aut}(\overline{T})$ is the semi-direct product 
$N_G(T)\rtimes D,$ where $N_G(T)$ is the normalizer of $T$ in $G$, and $D$ is the 
group of all automorphisms of the Dynkin diagram of $G$.
\end{theorem}

\begin{proof}
For notational convenience denote $$A\,=\,{\rm Aut}(\overline{T})\, .$$
Note that $\overline{T}$ is stable under the 
conjugation action of $N_G(T)$ on $\overline{G}.$ Let 
\begin{equation}\label{fan}
\widetilde{\Delta} \,\subset\, \mathfrak{t}_{\mathbb{R}} 
\end{equation}
be the fan of the toric variety $\overline{T}$. This 
$\widetilde\Delta$ consists of cones associated to the Weyl chambers (see \cite[p.~187, 6.1.6, 
Lemma]{BK}). Note that any automorphism $\sigma$ of the Dynkin Diagram associated to 
set $S \subset R$ of simple roots with respect to $(T,\, B)$ preserves the fan 
$\widetilde\Delta.$ Therefore, we have \cite[p.~47]{Co}
$$N_G(T)\rtimes D\, \subset\, A\,. $$

Next we will show that $N_G(T)\rtimes D \,=\,A$.

Since $\rho$ in \eqref{rh} is an isomorphism, it follows immediately that $T$ is a normal 
subgroup of $A$. Therefore, the intersection $T \bigcap g(T)$ is a $T$ stable open 
dense subset of $\overline{T}$ for every element $g \,\in\, A$. Consequently, the
open subset $T\, \subset\, \overline{T}$ is preserved by the natural action of $A$ 
on $\overline{T}$. Consequently, every automorphism $g\,\in\, A$ can be expressed as 
\begin{equation}\label{ed}
g\,=\,l_{t_0}h\, ,
\end{equation}
where $l_{t_0}$ is the left translation by some $t_0\,\in\, T$, 
and $h\,\in\, A$ satisfies the condition that $h(1)\,=\,1,$ with $1$ being the 
identity element of $T$.

By a result of Rosenlicht, the action of the $h$ (in \eqref{ed}) on $T$ is by group
automorphism (see 
\cite[p.~986, Theorem 3]{MR}). Therefore, $h$ gives an automorphism of $X(T)$, and hence 
$h$ gives an automorphism of $\mathfrak{t}_{\mathbb{R}}$. Since $T$ is left invariant 
under the action of $h$ the toric variety data of $\overline{T}$ is preserved by $h.$ Hence we see that the automorphism of $\mathfrak{t}_{\mathbb{R}}$
given by $h$ preserves the fan $\widetilde\Delta$ in \eqref{fan}. Since $\widetilde\Delta$ is given by the
Weyl chambers and its faces, we see that the induced action of $h$ on $X(T)$ leaves the root system $R$ of
$G$ in \eqref{root} invariant. Consequently, $h$ produces
an automorphism of the root system $R$.

On the other hand, the automorphism group ${\rm Aut}(R)$ of the root 
system $R$ is precisely
$$N_G(T)/T\rtimes D = W \rtimes D$$
(see \cite[p.~231, (A.8)]{HJ1}). 
\end{proof}

\begin{corollary}
The quotient group
${\rm Aut}(\overline{T})/{\rm Aut}^0(\overline{T})$ is isomorphic to
${\rm Aut}(R) = W\rtimes D.$ 
\end{corollary}

\begin{remark}\label{rm1}
The automorphism group $D$ is trivial
except for the types $A_{\ell}$ with $\ell \,\geq\, 2$, $D_{\ell}$ and $E_6$
(see \cite[p.~231, (A.8)]{HJ1}).
\end{remark}

\begin{remark}
We note that the structure of 
the automorphism group of a complete simplicial toric variety
is described by D. A. Cox (see \cite[p.~48, Corallary 4.7]{Co}).
\end{remark}

\section*{Acknowledgements}

We thank the referee for very helpful comments. 
The first--named author thanks the Institute of Mathematical Sciences for 
hospitality while this work was carried out. He also acknowledges the support of the 
J. C. Bose Fellowship. The second named author would like to thank the Infosys 
Foundation for the partial support.

\end{document}